\documentclass[10pt,a4paper]{article}            
\usepackage{amssymb,amsmath,amsfonts}
\usepackage{hyperref}
\usepackage{latexsym}
\newtheorem{theorem}{Theorem}[section]
\newtheorem{lemma}[theorem]{Lemma}

\newtheorem{proposition}[theorem]{Proposition}
\newtheorem{remark}[theorem]{Remark}
\numberwithin{equation}{section}
\newtheorem{definition}[theorem]{Definition}
\numberwithin{equation}{section}

\numberwithin{equation}{section}
\newtheorem{Basic assumptions}[theorem]{Basic assumptions}
\numberwithin{equation}{section}

\numberwithin{equation}{section}

\newenvironment{proof}{\smallskip\noindent\textit{Proof.}}{\hfill{$\Box$}
\medskip}

\newcount\quantno
\everydisplay{\quantno=0}\everycr{\quantno=0}
\newcommand\quant{\advance\quantno by1
                      \ifnum\quantno=1\qquad\else\quad\fi\forall }
\newcommand\itemno[1]{(\romannumeral #1)}

\newcommand\rest[1]{\kern-.1em
          \lower.5ex\hbox{$\scriptstyle #1$}\kern.05em}

\newcommand\set[1]{{\left\{#1\right\}}}

\renewcommand\mod[1]{\vert{#1}\vert}

\newcommand\norm[2]{{\Vert{#1}\Vert_{#2}}}

\newcommand\wrt{\,\text{\rm d}}

\newcommand\BC{\mathbb{C}}

\newcommand\BN{\mathbb{N}}
\newcommand\BR{\mathbb{R}}

\newcommand\la{\lambda}

\newcommand\funnyk{k\hbox to 0pt{\hss\phantom{g}}}

\newcommand\Fq[1]{F^q(#1)}
\newcommand\Fqi{F^q}

\newcommand\Finfty[1]{F^\infty(#1)}

\newcommand\Fcont[1]{F^C(#1)}
\newcommand\Fconti{F^{C}}
\newcommand\Fcontc[1]{F_c^C(#1)}
\newcommand\Fcontic{F_c^{C}}

\newcommand\lu[1]{L^1(#1)}

\newcommand\laq[1]{L^q(#1)}

\newcommand\hu[1]{H^1(#1)}

\newcommand\whH{\widehat{\phantom{G}}\hbox to 0pt{\hss $H$}}

\newcommand\emspace{\hbox to 6pt{\hss}}
\newcommand\ds{\displaystyle}

\newcommand\rmi{\hbox{\rm (i)}}
\newcommand\rmii{\hbox{\rm (ii)}}

\date{}

\begin{document}

\title{Equivalence of norms on \\
finite linear combinations of atoms
\thanks{Work partially supported by the
Italian Progetto Cofinanziato ``Analisi Armonica'', 2007.}
}

\author{Giancarlo Mauceri{\thanks{Dipartimento di Matematica, 
Universit\`a di Genova,
via Dodecaneso 35, 16146 Genova, Italy -- mauceri@dima.unige.it}}
\and Stefano Meda{\thanks{
Dipartimento di Matematica e Applicazioni, Universit\`a di Milano-Bicocca,
via R.~Cozzi 53, I-20125 Milano, Italy
-- stefano.meda@unimib.it
}}}

\maketitle

\begin{abstract}
Let $(M,\rho,\mu)$ be a space of homogeneous type and denote by \break  $\Fcontc{M}$ the space of finite linear combinations of continuous $(1,\infty)$-atoms with compact support. In this note we give a simple function theoretic proof of the equivalence on $\Fcontc{M}$ of the $H^1$-norm and the norm defined in terms of \emph{finite} linear combinations of atoms. The result holds also for the class of nondoubling metric measure spaces considered in previous works of the authors and A. Carbonaro.
\end{abstract}

\setcounter{section}{0}

\section[Introduction]
{Introduction}
Suppose that $q$ is in $(1,\infty]$.
A function $a$ in $\lu{\BR^d}$ is said to be a $(1,q)$-atom
if it is supported in a ball $B$ in $\BR^d$, and satisfies the following
conditions
$$
\int_{\BR^d} a(x) \wrt\la(x) = 0
\qquad \qquad\norm{a}{q} \leq \la(B)^{-1/q'},
$$
where $\la$ denotes the Lebesgue measure on $\BR^d$ and 
$q'$ is the conjugate index to $q$.
Denote by $\Fq{\BR^d}$ the vector space of all \emph{finite} 
linear combinations of $(1,q)$-atoms, endowed with 
the norm $\norm{\cdot}{F^q}$, defined as follows
$$
\norm{f}{F^q}
= \inf \Bigl\{ \sum_{j=1}^N \mod{\la_j}: f = \sum_{j=1}^N \la_j \, a_j,
\, \, 
\hbox{$a_j$ is a $(1,q)$-atom, $N\in\BN^+$} \Bigr\}.
$$
In \cite{MSV1}, the authors proved that if $q$ is finite,
then the $F^q$ norm and the restriction
to $\Fq{\BR^d}$ of the atomic $\hu{\BR^d}$ norm 
(defined just below (\ref{f: decomposition})) are equivalent.
The proof hinges on the atomic decomposition
and the maximal characterisation of $\hu{\BR^d}$, and is quite technical.
In the same paper, the authors proved a similar result for $\Fcont{\BR^d}$, a space
defined much as $\Fq{\BR^d}$, but with
continuous $(1,\infty)$-atoms in place $(1,q)$-atoms.
In \cite{GLY} the authors,
by adapting the techniques of \cite{MSV1}, succeeded to extend
these results to homogeneous 
spaces that satisfy an additional property, called property $RD$.
Apparently, there are serious obstructions in extending
this approach to all spaces of homogeneous type.
\par
F.~Ricci and J.~Verdera \cite{RV} complemented the analysis in 
\cite{MSV2} by proving that the dual of 
the completion of $\Finfty{\BR^d}$ is the direct sum of 
$BMO(\BR^d)$ and a nontrivial Banach space.
They observed also that variations of the main argument
in the proof of \cite[Thm~1]{RV} provide an alternative proof of 
the equivalence of the $\Fq{\BR^d}$ and the $\hu{\BR^d}$
norms on $\Fq{\BR^d}$ for $q<\infty$.  Their argument
revolves around the Gelfand--Naimark theory for
the commutative Banach algebra of all bounded functions
on $\BR^d$ that vanish at infinity.  
\par
The purpose of this paper is to show that an easy function theoretic approach  yields the equivalence  of the $\Fcont{M}$ and $\hu{M}$
norms on $\Fcont{M}$, 
in a very general setting, including all locally compact
spaces $M$ of homogeneous type and many interesting locally doubling
measured metric spaces, like those introduced in \cite{CMM1,CMM2}. We shall actually prove that the same result holds if we replace the space $\Fcont{M}$ with the space $\Fcontc{M}$ of finite linear combinations of continuous $(1,\infty)$-atoms \emph{with compact support}. Notice that  the space $\Fcontc{M}$ may be strictly smaller than $\Fcont{M}$. For instance, if $M$ is the open upper half-plane, endowed with the Euclidean metric and the Lebesgue measure, and $B$ is the ball with centre $(0,1)$ and radius $2$, there are $\Fcont{M}$-atoms supported in $B$ that do not have compact support. 
A similar approach works in the case where $q<\infty$
and gives the equivalence of the $\Fq{M}~$ and $\hu{M}$
norms on $\Fq{M}$ when $M$ is as above. In this case we do not even need to assume that $M$ is locally compact.
Our proof does not make use of the Gelfand--Naimark
theory for commutative Banach algebras, although our
main idea, which is to prove that the dual of $\Fcontc{M}$
is just $BMO(M)$, was inspired by \cite{RV}. 
\par
For the sake of simplicity, we shall restrict to spaces of infinite measure.
The case of spaces of finite measure may be treated similarly,
with very few changes due to the exceptional constant atom
and the slighty different definition of $BMO$. 
\par
We mention that our result is related to the extendability
of linear operators defined on finite linear combinations of atoms.   
Denote by $Y$ a Banach space and suppose that $q$ is in $(1,\infty]$.
We say that~$Y$ has the $q$-\emph{extension property}
if every $Y$-valued linear operator $T$, defined on \emph{finite}
linear combinations of $(1,q)$-atoms and satisfying the condition
\begin{equation} \label{f: extension property}
\sup\{ \norm{T a}{Y}: \hbox{$a$ is a $(1,q)$-atom} \} < \infty,
\end{equation}
extends to a bounded operator from $\hu{M}$ to $Y$. 
Similarly, 
we say that $Y$ has the \emph{continuous $\infty$-extension property}
if $T$ satisfies the condition above with continuous
$(1,\infty)$-atoms in place of $(1,q)$-atoms.

On the one hand, no Banach space $Y$ has
the $\infty$-extension property. 
Indeed, a direct consequence of a recent result of
M.~Bownik \cite{Bow} is that for every Banach space $Y$
there exists a $Y$-valued operator $B$, defined on finite linear combinations
of $(1,\infty)$-atoms that satisfies
\begin{equation} \label{f: unif bound}
\sup\{ \norm{B a}{Y}: \hbox{$a$ is a $(1,\infty)$-atom} \} < \infty,
\end{equation}
but does not admit an extension to a bounded operator from $\hu{\BR^d}$
to $Y$.  
\par
On the other hand, every Banach space $Y$ has the $q$-extension property
for all $q$ in $(1,\infty)$, and the continuous $\infty$-extension
property.  
This is proved in \cite{YZ} in the case where $q=2$,
and, independently, in \cite{MSV1} for all $q$ in $(1,\infty)$ 
and in the continuous $\infty$ case.  
\par
Our analysis is limited to the Hardy space $\hu{M}$,
because the applications we have in mind are to the space $H^1$ rather
than to $H^p$ with $p$ in $(0,1)$.  We leave the investigation of the 
interesting case $p$ in $(0,1)$ to further studies. 
An elegant analysis of the case $p\leq 1$ in the Euclidean setting
may be found in \cite{RV}.  
See also the papers \cite{MSV1, MSV2, YZ}.

\thanks{The authors thank Detlef M\"uller for a useful suggestion that led to a simplification in the proof of Theorem \ref{t: main}}.

\section[Notation and background]{Notation and background information}
\label{s: Notation}
Suppose that $(M, \rho, \mu)$ is a space of homogeneous type. 
In particular, $\rho$ is a quasi-distance on $M$
and $\mu$ is a regular Borel measure on $M$.  We shall
assume that $\mu(M) = \infty$ and 
that $0<\mu(B)<\infty$ for all balls in $M$.
We refer the reader to \cite{CW} for any unexplained notation
concerning spaces of homogeneous type. 

\begin{definition} \label{d: standard atom}
\emph{Suppose that $q$ is in $(1,\infty]$.  A $(1,q)$-\emph{atom} 
$a$ associated to a ball $B$ is a function in $L^q(M)$ 
supported in $B$
with the following properties:
\begin{enumerate}
\item[\itemno1]
$\norm{a}{q}  \leq \mu (B)^{-1/q'}$, where $q'$ denotes the index
conjugate to $q$;
\item[\itemno2]
$\ds \int_B a \wrt \mu  = 0$.
\end{enumerate}}
\end{definition}
\par
\noindent
The \emph{Hardy space} $H^{1,q}({M})$ is the 
space of all functions~$g$ in $\lu{M}$
that admit a decomposition of the form
\begin{equation} \label{f: decomposition}
g = \sum_{k=1}^\infty \la_k \, a_k,
\end{equation}
where $a_k$ is a $(1,q)$-atom and $\sum_{k=1}^\infty \mod{\la_k} < \infty$.
The norm $\norm{g}{H^{1,q}}$
of $g$ is the infimum of $\sum_{k=1}^\infty \mod{\la_k}$
over all decompositions (\ref{f: decomposition})
of $g$. All the $\norm{\phantom{\cdot}}{H^{1,q}}$ norms are equivalent. Hence the vector spaces $H^{1,q}(M)$ coincide. Henceforth we shall simply denote by $\norm{\phantom{\cdot}}{H^{1}}$ the norm $\norm{\phantom{\cdot}}{H^{1,\infty}}$ and by $H^1(M)$ the space $H^{1,\infty}(M)$.\par
The dual of the space $H^1(M)$ can be identified with the quotient of the space $BMO(M)$ of functions of bounded mean oscillation modulo the constants, endowed with any of the equivalent norms 
$$
\norm{f}{BMO^{q'}}=\sup_{B}\left(\frac{1}{\mu(B)}\int_B\mod{f-\langle f\rangle_{\!B}}^{q'}\wrt \mu\right)^{1/q'},
$$
where $1\le q'<\infty$, $B$ ranges over all balls in $M$ and $\langle f\rangle_B$ denotes the average of $f$ over $B$. We shall denote  the norm  $\norm{\phantom{\cdot}}{BMO^1}$ simply by $\norm{\phantom{\cdot}}{BMO}$.
\par
For each $q$ in $(1,\infty)$ we denote by
$\Fq{M}$ the normed space of all finite
linear combinations of $(1,q)$-atoms.   The norm
on $\Fq{M}$ is defined as follows
\begin{equation} \label{f: huqfin norm}
\norm{f}{F^q}
= \inf \Bigl\{ \sum_{j=1}^N \mod{\la_j}: f = \sum_{j=1}^N \la_j \, a_j,
\, \, 
\hbox{$a_j$ is a $(1,q)$-atom, $N\in\BN^+$} \Bigr\}.
\end{equation}
We denote by $\Fcont{M}$
the normed space of all finite linear combinations
of \emph{continuous} $(1,\infty)$-atoms and by $\Fcontc{M}$ the space of all finite linear combinations
of {continuous} $(1,\infty)$-atoms with \emph{compact support}. We shall refer to these atoms as $\Fcont{M}$-atoms and $\Fcontc{M}$-atoms, respectively. The 
$\Fcont{M}$ and the  $\Fcontc{M}$ norms are defined as the $\Fq{M}$
norm above, but with $\Fcont{M}$-atoms  and $\Fcontc{M}$-atoms, respectively, in place
of $(1,q)$-atoms.
Obviously 
\begin{equation} \label{f: hu-Ffin} 
\norm{f}{H^{1}} 
\leq  \norm{f}{\Fconti}\leq  \norm{f}{\Fcontic}
\quant f \in \Fcontc{M},
\end{equation}
and \begin{equation} \label{f: hu-huqfin} 
\norm{f}{H^{1,q}} 
\leq  \norm{f}{F^q}
\quant f \in \Fq{M}.
\end{equation}

 For each \emph{open} ball $B$ in $M$, 
we denote by  $C_c(B)$  the subspace of all  functions with compact support contained in $B$, \emph{endowed with the supremum norm}. We assume that $M$ is locally compact, so that  the  dual $C_c(B)^*$ of $C_c(B)$ is the space of complex Borel measures on  $B$.

 We denote by $C_c(M)$ the space of all continuous functions on $M$ with compact support endowed with the finest linear topology such that all the inclusion maps $C_c(B)\subset C_c(M)$ are continuous. The dual of $C_c(M)$ is the space $C_c(M)^*$ of all linear functionals $\nu$ on $C_c(M)$, whose restriction $\nu_B$ to $C_c(B)$ is a continuous linear functional on $C_c(B)$  for every ball $B$.  Thus $\nu_B$ is a complex Borel measure on $B$, and we shall denote by $\mod{\nu_B}$ its total variation. If $\nu\in C_c(M)^*$, then its total variation is the nonnegative $\sigma$-finite Radon measure on $M$, defined by
$$
\mod{\nu}(E)=\sup_B \mod{\nu_B}(E)
$$
for all Borel subset $E$ of $M$. By applying the Radon-Nikodym theorem to the restrictions of $\nu$ to a sequence of balls that invades $M$, it is easy to see that there exists a Borel function $h$ such that $\nu(\phi)=\int_M \phi\, h\wrt\mod{\nu}$ for all $\phi$ in $C_c(M)$,  and $\mod{h}=1$ on $M$. Henceforth, to simplify notation, we shall denote by $\nu$ all the measures $\nu_B$. \par We denote by $C_{c,0}(M)$ and $C_{c,0}(B)$ the subspaces of $C_{c}(M)$ and $C_{c}(B)$, respectively, of all functions such that $\int_Mf \wrt \mu=0$. Note that the $\Fcontc{M}$, as a vector space, coincides with the space $C_{c,0}(M)$.

\begin{remark}\label{ann}
 The annihilator of $C_{c,0}(M)$ 
in $C_c(M)^*$ is $\BC\mu$. 
Indeed, $\mu$ annihilates $C_{c,0}(M)$ 
by definition and, if $\nu$ is a linear functional on $C_c(M)$ that annihilates  $C_{c,0}(M)$, 
then $\nu=\alpha\mu$ for some $\alpha\in\BC$, 
because  $\ker(\mu)\subset\ker(\nu)$. 
Thus, the dual of $C_{c,0}(M)$ is isomorphic to the quotient $C_c(M)^*/\set{\BC\mu}$. Similarly, the dual of $C_{c,0}(B)$ is isomorphic to the quotient $C_c(B)^*/\set{\BC\mu}$. 
\end{remark}
\begin{definition}\label{mbmo}
We denote by $\langle\nu\rangle_{\!B}$ the quotient $\nu(B)/\mu(B)$. We say that a linear functional $\nu\in  C_c(M)^*$ has bounded mean oscillation if
\begin{equation}\label{*bmo}
\norm{\nu}{BMO}=\sup_B \frac{1}{\mu(B)} \left| \nu-\langle\nu\rangle_{\!B}\,\mu\right|(B)<\infty.
\end{equation}
\end{definition}
\begin{remark}\label{rembmo1}
Observe that for any ball $B$
\begin{equation}\label{f: equiv}
\inf_\alpha\left|\nu-\alpha\mu\right|(B)\le \left|\nu-\langle\nu\rangle_{\!B}\,\mu\right|(B)\le 2\, \inf_\alpha\left|\nu-\alpha\mu\right|(B)
\end{equation}
Indeed, the first inequality is obvious. To prove the second, add 
and subtract $\alpha\mu$  from $\nu-\langle\nu\rangle_{\!B}\,\mu$,
 apply the triangle inequality  and observe that \break $\left|\alpha -\langle\nu\rangle_{\!B}\right|\mu(B)\le\mod{\alpha\mu-\nu}(B)$.
Thus,
in Definition \ref{mbmo} we may allow arbitrary 
constants $c_B$ in place of $\langle\nu\rangle_{\!B}$. 
\end{remark}
\begin{remark}\label{rembmo2}
Since $\big|\mod{\nu}-\mod{\langle\nu\rangle_{\!B}}\mu\big|(B)\le \mod{\nu-\langle\nu\rangle_{\!B}\mu}(B)$ for every ball $B$, by the previous remark $\mod{\nu}$ has bounded mean oscillation whenever $\nu$ does. 
\end{remark}

\begin{lemma}\label{mbmo=bmo}
If $\nu\in C_c^*(M)$ has bounded mean oscillation
then there exists a function $f$ in $BMO(M)$ such that $\nu=f\mu$. Moreover $\norm{\nu}{BMO}=\norm{f}{BMO}$.
\end{lemma}
\begin{proof}
By Remark \ref{rembmo2} the Radon measure $\mod{\nu}$ has bounded mean oscillation.
Let $\mod{\nu}=\mod{\nu}_{a}+\mod{\nu}_{s}$ be the Lebesgue decomposition  of $\mod{\nu}$ into its absolutely continuous and singular parts with respect to $\mu$. 
Denote by $N$ a set that carries $\mod{\nu}_{s}$. Then $\mod{\nu}_{s}(E\cap N)=\mod{\nu}_{s}(E)$ for every Borel subset 
$E$ of $M$ and
$\mu(N)=\mod{\nu}_{a}(N)=0$. Let $B$ be an open ball in $M$. 
Then, since both $\mod{\nu}_{a}(B\cap N)$ and $\mu(B\cap N)$ are
 zero,
 \begin{align*}
\mod{\nu}_{s}(B\cap N)&=\mod{\nu}(B\cap N) -
\langle\mod{\nu}\rangle_{\!B}\,\mu(B\cap N)\\ 
&\le \big|\mod{\nu}-\langle\mod{\nu}\rangle_{\!B}\,\mu\big|(B) \\ 
&\le  \norm{\mod{\nu}}{BMO}\,\mu(B).
\end{align*}
Since
$\mod{\nu}_{s}(B)=
 \mod{\nu}_{s}(B\cap N)$, we have proved that for all open balls $B$ in $M$
 \begin{equation}\label{f: nusB}
\mod{\nu}_{s}(B)\le \norm{\mod{\nu}}{BMO}\,\mu(B).
\end{equation}
Now, (\ref{f: nusB}) implies that there exists a constant $C$ such that
$$
\mod{\nu}_{s}(U)\le C\, \norm{\mod{\nu}}{BMO}\,\mu(U).
$$
for every bounded open set $U$ in $M$. 
Indeed, by the basic covering theorem \cite[Thm 1.2]{H} there  exists a family $\{B^j\}$
of mutually disjoint open balls such~that
$$
\bigcup_j B^j 
\subseteq U \subseteq
\bigcup_j (5 B^j).
$$
The family $B^j$ is countable, because $\mu(U)<\infty$.
Then 
\begin{align}\label{doub}
\mod{\nu}_{s}(U)
& \leq  \sum_j  \mod{\nu}_{s}(5B^j) \nonumber \\
& \leq\,  \norm{\mod{\nu}}{BMO} \, \sum_j  \mu(5B^j) \nonumber\\
& \leq C\, \norm{\mod{\nu}}{BMO}  \, \sum_j  \mu(B^j) \\
& \leq  C\, \norm{\mod{\nu}}{BMO}  \, \mu(U), \nonumber
\end{align}
as required. 
Since $\mod{\nu}_{s}$ is a $\sigma$-finite Radon measure, it is outer regular on all Borel sets,
whence
$$
\mod{\nu}_{s}(E)
\leq C\,\norm{\mod{\nu}}{BMO}  \, \mu(E)
$$
for every bounded Borel set $E$ in $M$. 
This inequality implies that the singular measure 
$\mod{\nu}_{s}$ is also absolutely continuous
with respect to $\mu$.   Hence $\mod{\nu}_{s}=0$, and 
$\mod{\nu}$ is absolutely continuous with respect to $\mu$. Therefore there exists a function $f$ such that $\nu=f\mu$. It is straightforward to verify that $\norm{f}{BMO}=\norm{\nu}{BMO}$.
\end{proof}
\begin{remark}\label{r: locdoub}
Note that the assumption that the measure $\mu$ is doubling is used only in the proof of the inequality (\ref{doub}). Actually, it suffices to assume that $\mu$  satisfies the doubling condition only on balls of radius at most $1$.
\end{remark}
\section[The main result]{The main result} \label{s: The main result}

The equivalence of the norms on $H^1(M)$ and $F^C_c(M)$ will be an easy consequence of Proposition \ref{FCcduale} below, which shows that the embeddings of $F^C_c(M)$ and of $\Fq{M}$ in $H^1(M)$ induce isomorphisms of the duals. 

If $\nu\in H^1(M)^*$ we denote by $\nu^\flat$ its restriction to $F^C_c(M)$.
\begin{proposition}\label{FCcduale} The following hold
\begin{itemize}
\item[\rmi] If $M$ is locally compact then the map $\nu\mapsto\nu^\flat$ is an isomorphism of $H^1(M)^*$ onto $F^C_c(M)^*$.
\item[\rmii] If $1<q<\infty$, then a similar statement holds with $\Fcontic(M)$ replaced by $\Fq{M}$ (without assuming that $M$ be locally compact).
\end{itemize}

\end{proposition}
\begin{proof} First we prove \rmi. 
By (\ref{f: hu-Ffin}) the linear functional $\nu^\flat$ is in $F^C_c(M)^*$ and
\begin{equation}\label{1stineq}
\norm{\nu^\flat}{F^C_c(M)^*}\le \norm{\nu}{H^1(M)^*}.
\end{equation}
Next we prove that the map $\nu\mapsto \nu^\flat$ is injective. Suppose that $\nu^\flat=0$. Let $f$ be a function in $BMO(M)$ that represents $\nu$, i.e. such that $\nu(\phi)=\int_Mf\phi\wrt\mu$ for all finite linear combinations of atoms $\phi$. Since $F^C_c(M)=C_{c,0}(M)$,
$$
\int_M f\phi\wrt\mu=\nu^\flat(\phi)=0 \qquad\forall \phi\in C_{c,0}(M).
$$
Thus $f$ is constant and $\nu=0$. Therefore $\nu\mapsto\nu^\flat$ is injective.\par
Next, we show that $\nu\mapsto \nu^\flat$ is surjective. 
Let $\lambda$ be in $F^C_c(M)^*$.
Since $F^C_c(M)=C_{c,0}(M)$ has codimension $1$ in $C_c(M)$, the functional $\lambda$ extends to a continuous linear functional $\nu$ on $C_c(M)$. We prove that $\nu$ has bounded mean oscillation and 
\begin{equation}\label{leq2}
\norm{\nu}{BMO}\le \,2\, \norm{\lambda}{F^C_c(M)^*}.
\end{equation}
For every ball $B$, the restriction of $\nu$ to the subspace $C_c(B)$ is a complex Borel measure $\nu_B$ on $B$, which we shall henceforth denote simply by $\nu$ to simplify notation. For every function $\phi$ in $C_{c,0}(B)$, the function $\phi/(\mu(B)\norm{\phi}{\infty})$ is an atom supported in $B$. Thus
\begin{equation}\label{nula}
\mod{\nu(\phi)}=\mod{\lambda(\phi)}\le \norm{\lambda}{F^C_c(M)^*}\ \mu(B) \ \norm{\phi}{\infty}\qquad\forall \phi\in C_{c,0}(B).
\end{equation}
Since $C_{c,0}(B)^*$ is isometric to the quotient $C_c(B)^*/\set{\BC\mu}$,
\begin{align}
\inf_\alpha \left|\nu-\alpha\mu\right|(B) \nonumber &=\norm{\nu}{C_{c,0}(B)^*} \\
&=\sup\set{\mod{\nu(\phi)}: \phi \in C_{c,0}(B), \norm{\phi}{\infty}\le 1} \label{nq}   \nonumber\\
&\le\, \norm{\lambda}{F^C_c(M)^*}\ \mu(B), \nonumber
\end{align}
by (\ref{nula}). Then, by Remark \ref{rembmo1},
\begin{equation}\label{f: MBMO}
\frac{1}{\mu(B)} \, 
\left| \nu-\langle\nu\rangle_{\!B}\,\mu\right|(B)\leq 2\,\norm{\lambda}{F^C_c(M)^*}.
\end{equation} 
Therefore $\nu$ has bounded mean oscillation and satisfies (\ref{leq2}).
By Lemma \ref{mbmo=bmo} there exists a function $f$ in $BMO(M)$ such that $\nu=f\mu$, and $\norm{f}{BMO}=\norm{\nu}{BMO}$. Hence $\nu$ extends to a continuous linear functional on $H^1(M)$, whose restriction to $F^C_c(M)$ is $\lambda$, i.e $\lambda=\nu^\flat$. Moreover, by (\ref{1stineq}) and (\ref{leq2}),
$$
\norm{\nu^\flat}{F^C_c(M)^*}\le \norm{\nu}{H^1(M)^*}\le 2\,\norm{\nu^\flat}{F^C_c(M)^*}.
$$
This concludes the proof of the first part of the theorem.\par
The proof of \rmii\ follows the same lines.  We simply replace 
$C_{c}(B)$ by the space $L^q(B)$ of all functions in 
$\laq{M}$ that are supported in $B$ and have integral zero and $C_c(M)$ by the space $L^q_c(M)$ of all functions in $L^q(M)$ with compact support, endowed with finest linear topology for which all the inclusions $L^q(B)\subset L^q_c(M)$ are continuous. The dual of $L^q_c(M)$ is the space $L^{q'}_{\text{loc}}(M)$, where $1/q+1/q'=1$.
Note that here we do not need  the assumption that $M$ be locally compact,
which is used, instead, in the proof of \rmi\ 
to identify the dual of the space of 
continuous functions with compact support contained in an open ball $B$
with the space of all complex Borel measures on $B$.
\end{proof} 

\begin{theorem} \label{t: main}

\begin{itemize}
\item[\rmi]
 If $M$ is locally compact, then   there exists a constant $C$ such that 
\begin{equation} \label{f: second main}
\norm{f}{H^1}  
\le \norm{f}{\Fcontic} \le C\ \norm{f}{H^1} 
\quant f \in \Fcontc{M}.  
\end{equation}
\item[\rmii]
For each $q$ in $(1,\infty)$  there exists a constant $C$ such that
\begin{equation} \label{f: first main}
\norm{f}{H^{1,q}}  
\le \norm{f}{\Fqi}\le C\ \norm{f}{H^{1,q}}  
\quant f \in\Fq{M}.  
\end{equation}
\end{itemize}
\end{theorem}
\begin{proof} 
We have already 
remarked that $\norm{f}{H^1}\le \norm{f}{\Fcontic}$ for all $f\in
\Fcontic(M)$, (see  (\ref{f: hu-Ffin})).  By Proposition \ref{FCcduale},  for every $f\in \Fcontic(M)$ 
\begin{align*}
\norm{f}{\Fcontic}&=\sup\set{\mod{\lambda(f)}:\lambda\in \Fcontic(M)^*, \norm{\lambda}{\Fcontic(M)^*}\le1} \\ 
&\le  \sup\set{\left|\Lambda(f)\right|: \Lambda\in H^1(M)^*,\ \norm{\Lambda}{H^1(M)^*}\le C}\\
&\le C\, \norm{f}{H^1}. 
\end{align*}
\par
A similar argument shows that $\norm{f}{\Fqi}\le C\,\norm{f}{H^{1,q}}$ for all $f$ in $\Fq{M}$. 
\end{proof}
\par\noindent 

\begin{remark}\label{r: FC} 
If $M$ is locally compact then the space $F^C_c(M)$  is dense in $H^1(M)$. This is a straightforward consequence of Proposition \ref{FCcduale} and of the Hahn-Banach theorem. 
Note that (\ref{f: hu-Ffin}), (\ref{f: second main}) and the density of $\Fcontc{M}$ in $\hu{M}$ imply that the $\Fcont{M}$ and the $\hu{M}$ norms are equivalent on $\Fcont{M}$. Thus $H^1(M)$ is the completion also  of $\Fcont{M}$ in the norm of the latter space. Similar statements hold with $F^C_c(M)$ replaced by $\Fq{M}$, without the assumption that $M$ be locally compact.
\end{remark} 

\section{Locally doubling measured metric spaces}

As mentioned in the introduction, the main result
we presented in the last section in the case of 
spaces of homogenous type, may be generalised to a variety of
settings.   
In this section we describe the generalisation
to the case of the atomic Hardy spaces on certain measured
metric spaces introduced in \cite{CMM1,CMM2}.  
We restrict to the case where the space has infinite measure.
Again, slight modifications will also cover the
finite measure case.  
\par
Suppose that $(M,\rho,\mu)$ is a measure metric space as in \cite[Section~2.1]{CMM1}. More precisely, we assume that $(M,\rho)$ satisfies the ``approximate midpoint property" and that the measure $\mu$ is locally doubling, but for the purposes of the present paper, we do not need to assume that $\mu$ satisfies the ``isoperimetric property".  
\par
Denote by $\hu{M}$ the atomic Hardy
space defined in \cite{CMM1}.
We recall that the definition of a $\hu{M}$-atom is exactly as
in the case of spaces of homogeneous type, but, unlike in the
classical case,  $H^1({M})$ is the 
space of all functions~$g$ in $\lu{M}$
that admit a decomposition of the form
\begin{equation} \label{f: decomposition I}
g = \sum_{k=1}^\infty \la_k \, a_k,
\end{equation}
where $a_k$ is a $\hu{M}$-atom \emph{supported in a ball $B$ of radius 
at most $1$},
and $\sum_{k=1}^\infty \mod{\la_k} < \infty$.
The norm $\norm{g}{H^1}$
of $g$ is the infimum of $\sum_{k=1}^\infty \mod{\la_k}$
over all decompositions (\ref{f: decomposition I}) of $g$.
Similarly,  we may define
the spaces $\Fq{M}$,  $\Fcont{M}$ and $\Fcontc{M}$, 
using only
atoms supported in balls of radius at most $1$.

Straightforward adaptations of the arguments of the previous section
yield the following result (see Remark \ref{r: locdoub}). 

\begin{theorem}  \label{t: main 2}
The following hold
\begin{itemize}
\item[\rmi]
if $M$ is locally compact, then the space $\Fcontc{M}$ is dense in $\hu{M}$ and there exists a constant $C$ such that
$$
\norm{f}{H^1}  \le  \norm{f}{\Fcontic}\le\,C\,\norm{f}{H^1} \quant f \in \Fcontc{M}.
$$
\item[\rmii]
If $1<q<\infty$, then $\Fq{M}$ is dense in $\hu{M}$ and there exists a constant $C$ such that
$$
\norm{f}{H^{1,q}}  
\le \norm{f}{\Fqi} \le \,C\,  \norm{f}{H^{1,q}} \quant f \in\Fq{M}
.$$
\end{itemize}
\end{theorem}


\begin{thebibliography}{GMMST}

\bibitem[Bow]{Bow} M. Bownik, Boundedness of operators
on Hardy spaces via atomic decomposition,
\emph{Proc. Amer. Math. Soc.}, \textbf{133} (2005), 3535--3542. 

\bibitem[CMM1]{CMM1} A. Carbonaro, G. Mauceri and S. Meda, 
$H^1$, $BMO$ and singular integrals for certain metric measure
spaces, 
\emph{Ann. Sc. Norm. Super. Pisa Cl. Sci.} (5) 8 (2009), n. 3, 543-582.

\bibitem[CMM2]{CMM2} A. Carbonaro, G. Mauceri and S. Meda, 
$H^1$ and $BMO$ for certain locally doubling metric measure
spaces of finite measure,  {\tt arXiv:0811.0100v1 [math.FA]}
to appear in \emph{Coll. Math.}

\bibitem[CW]{CW} R.R. Coifman and G. Weiss,  Extensions of Hardy
spaces and their use in analysis,
\emph{Bull. Amer. Math. Soc.} \textbf{83} (1977), 569--645.

\bibitem[GLY]{GLY} L. Grafakos, L. Liu and D.~Yang,
Maximal function characterization of Hardy spaces on $RD$-spaces and
their applications, \emph{Sci. China Ser. A} \textbf {51} (2008), 2253--2284.  

\bibitem[H]{H} J. Heinonen, \emph{Lectures on Analysis on Metric Spaces}, Universitext, Springer-Verlag, New York, 2001. 

\bibitem[MSV1]{MSV1} S. Meda, P. Sj\"ogren and M. Vallarino, 
On the $H^1$--$L^1$ boundedness of operators,
\emph{Proc. Amer. Math. Soc.} \textbf{136} (2008), 2921--2931.

\bibitem[MSV2]{MSV2} S. Meda, P. Sj\"ogren and M. Vallarino, Atomic decompositions and operators on Hardy spaces, \emph{Rev. Un. Mat. Argentina}, \textbf{50} (2009), 15--22.


\bibitem[RV]{RV} F. Ricci and J. Verdera,
Duality in spaces of finite linear combinations
of atoms, to appear in \emph{Trans. Amer. Math. Soc.}



\bibitem[YZ]{YZ} D.~Yang and Y.~Zhou,
A boundedness criterion via atoms for linear operators in Hardy
spaces, \emph{Constr. Approx.} \textbf{29} (2009), 207--218.
\end{thebibliography}
\end{document}